\documentclass[11pt]{article}
\usepackage{indentfirst}
\usepackage{latexsym}
\usepackage{hyperref}
\usepackage{graphicx}
\usepackage{mathrsfs}
\usepackage{bookmark}
\usepackage{enumerate}
\usepackage{amsfonts,amsthm,amssymb,mathtools}
\usepackage{amsmath}
\usepackage{latexsym, euscript, epic, eepic, color}
\usepackage{amssymb}

\usepackage{enumerate}
\usepackage[all]{xy}
\usepackage{setspace}
\allowdisplaybreaks

\usepackage{epstopdf}
\numberwithin{equation}{section}
\newtheorem{theorem}{Theorem}[section]
\newtheorem{corollary}[theorem]{Corollary}
\newtheorem{definition}[theorem]{Definition}

\newtheorem{lemma}[theorem]{Lemma}

 \allowdisplaybreaks

\begin{document}
\baselineskip=16pt

\title{Sharp bounds on the zeroth-order general Randi\'c index of trees in terms of domination number \footnote{This work was supported by the National Natural Science Foundation of China [61773020] and Postgraduate Scientific Research  Innovation Project of Hunan Province [CX20200033]. Corresponding author: Jianping Li (lijianping65@nudt.edu.cn).
 }
}

\author{Chang Liu, Jianping Li   \\
	\small  College of Liberal Arts and Sciences, National University of Defense Technology, \\
	\small  Changsha, China, 410073.\\
}

\date{\today}

\maketitle

\begin{abstract}
The zeroth-order general Randi\'c index of graph $G=(V_G,E_G)$, denoted by $^0R_{\alpha}(G)$, is the sum of items $(d_{v})^{\alpha}$ over all vertices $v\in V_G$, where $\alpha$ is a pertinently chosen real number. In this paper, we obtain the sharp upper and lower bounds on $^0R_{\alpha}$ of trees with a domination number $\gamma$, in intervals $\alpha\in(-\infty,0)\cup(1,\infty)$ and $\alpha\in(0,1)$, respectively. The corresponding extremal graphs of these bounds are also characterized.
\\[2pt]
\textbf{AMS Subject Classification:} 05C50; 05C35; 05C69\\[2pt]
\textbf{Keywords:} The zeroth-order general Randi\'c index; Extremal trees; Domination number
\end{abstract}

\section{Introduction}
Let $G$ be a graph with vertex vertex $V_{G}$ and edge set $E_{G}$. The general Randi\'c index is defined as 
\begin{equation*}
	R_{\alpha}=R_{\alpha}(G)=\sum_{uv\in E_G}(d_ud_v)^{\alpha},
\end{equation*}
where $d_v$ denotes the degree of a vertex $v\in V(G)$, and $\alpha$ is an arbitrary real number. It's widely known that $R_{-\frac{1}{2}}$, i.e., the Randi\'c index in original sense, was introduced by the chemist Milan Randi\'c \cite{rand} under the name \textit{connectivity index} or \textit{branching index} in 1975,  which has a good correlation with a variety of physico-chemical properties of alkanes, such as enthalpy of formation, boiling point,  parameters in the Antoine equation, surface area and solubility in water, etc. In the past 30 to 40 years, the Randi\'c index has been widely utilized in physics, chemistry, biology, and complex networks \cite{Deh,rand2}, and many interesting mathematical properties have been obtained \cite{Del,Liu,Li}. In 1998, Bollob\'as and Erd\"os \cite{Boll} generalized this index by replacing $-\frac{1}{2}$ with a real number $\alpha$, and called it the general Randi\'c index, denoted by $R_{\alpha}=R_{\alpha}(G)$. 

Moreover, there are also many variants of Randi\'c index \cite{Dvo,Knor,Shi}. In \cite{Kier2}, Kier and Hall proposed the zeroth-order Randi\'c index, denoted by $^0R$. The explicit formula of $^0R$ is 
\begin{equation*}
	^0R=^0R(G)=\sum_{v\in V_G}(d_v)^{-\frac{1}{2}}.
\end{equation*}
In some bibliographies, $^0R$ is also called the modified first Zagreb index ($^mM_1$). Pavlovi\'c \cite{Pav} determined the extremal $(n,m)$-graphs of $^0R$ with maximum value. Almost at the same time, Lang et al. considered similar problems in \cite{Lang} for the first Zagreb index ($M_1$), which is defined as
\begin{equation*}
	M_1=M_1(G)=\sum_{v\in V_G}(d_v)^2.
\end{equation*}

In 2005, Li and Zheng \cite{Lix1} constructed the zeroth-order general Randi\'c index, written $^0R_{\alpha}$, is the sum of items $(d_v)^{\alpha}$ over all vertices $v\in V_G$, where $\alpha$ is an pertinently chosen real number. Note that $^0R_{-\frac{1}{2}}={^0R}={^mM_1}$, and $^0R_{2}=M_1$ in the mathematical sense. For the zeroth-order general Randi\'c index of trees, Li and Zhao \cite{Lix2} determined the first three maximum and minimum values with exponent $\alpha_0,-\alpha_0,\dfrac{1}{\alpha_0},\dfrac{1}{\alpha_0}$, where $\alpha_0\geq 2$ is an integer. In 2007, Hu et al. \cite{Hu} investigated connected $(n,m)$-graphs with extremal values of $^0R_{\alpha}$. Two years later, in \cite{Pav2}, Pavlovi\'c et al. corrected some errors in the work of Hu et al.

Recently, the relationships between Randi\'c-type indices and domination number has attracted much attention of many researchers. In 2016, Borovi\'canin and Furtula \cite{Boro2} gave the precise upper and lower bounds on the first Zagreb index ($M_1$) of trees in terms of domination number and characterized the corresponding extremal trees. Later, Bermudo et al. \cite{Berm} and Liu et al. \cite{CLiu} answered the same question regarding the Randi\'c index ($R_{-\frac{1}{2}}$) and the modified first Zagreb index ($^mM_1$), respectively. Motivate by \cite{Boro2,Berm,CLiu}, in this paper, we intend to establish connections between the zeroth-order general Randi\'c index of trees and domination number.

For convenience, we first introduce some graph-theoretic terminology and notions. The number of vertices and edges of graph $G$ are called the order and size of $G$, respectively. For each $v\in V_G$, the set of neighbours of this vertex is denoted by $N(v)=\left\lbrace u\in V_G|~uv\in E_G \right\rbrace $. A vertex $v$ for which $d_v=1$ is called a pendent vertex or a leaf vertex. The maximum vertex degree in $G$ is denote by $\Delta(G)$. The diameter of a tree is the longest path between two pendent vertices. The dominating set of graph $G$ is a vertex subset in $V_G$ such that every vertex in $V_G\setminus D$ is adjacent at least one vertex in $D$. A subset $D$ is called minimum dominating set of $G$ if $D$ contains least vertices among all dominating sets. Domination number $\gamma$ is defined as $\gamma=|D|$.

Based on the above consideration, the structure of this paper is arranged as below. In Section 2, we prove a fundamental lemma and simplify the mathematical formula of several bounds on $^0R_{\alpha}$. Then in Section 3 and 4, sharp upper and lower bounds on $^0R_{\alpha}$ of trees with a given domination number for $\alpha\in(-\infty,0)\cup(1,\infty)$ and $\alpha\in(0,1)$ are obtained, respectively. Furthermore, the corresponding extremal trees are characterized.
\section{Preliminaries}
Now, we present a basic lemma, and then show the simplified mathematical formula of bounds on $^0R_{\alpha}$.
\begin{lemma}\label{lem1}
	Let the function $f(x_1,x_2,\cdots,x_k)=x_{1}^{\alpha}+x_{2}^{\alpha}+\cdots+x_{k}^{\alpha}$, where $\left\lbrace x_1,x_2,\cdots,x_k\right\rbrace $ are all positive integers, and $\sum_{i=1}^{k}x_i=N$ be a fixed positive integer. If there exist $x_i$ and $x_j$ such that $x_i-x_j\geq2$, $i,j\in\left\lbrace1,2,\cdots,k \right\rbrace $ and $i\ne j$, then
	\begin{itemize}
		\item[\textnormal{(\romannumeral1)}]$f(x_1,\cdots,x_i,\cdots,x_j,\cdots,x_k)<f(x_1,\cdots,x_i-1,\cdots,x_j+1,\cdots,x_k)$, for $\alpha\in(0,1)$,
		\item[\textnormal{(\romannumeral2)}]$f(x_1,\cdots,x_i,\cdots,x_j,\cdots,x_k)>f(x_1,\cdots,x_i-1,\cdots,x_j+1,\cdots,x_k)$, for $\alpha\in(-\infty,0)\cup(1,\infty)$.
	\end{itemize}
\end{lemma}
\begin{proof} From the definition of function $f(x_1,x_2,\cdots,x_k)$, we consider the following function $g(x_i,x_j,\alpha)$,
\begin{align*}
	g(x_i,x_j,\alpha)&=f(x_1,\cdots,x_i,\cdots,x_j,\cdots,x_k)-f(x_1,\cdots,x_i-1,\cdots,x_j+1,\cdots,x_k)\\
	&=x_{i}^{\alpha}+x_{j}^{\alpha}-(x_i-1)^{\alpha}-(x_j+1)^{\alpha}.
\end{align*}
Note that $x_i-x_j\geq2$, we let $x_i=x_j+r$, where $r\geq2$ is a integer. Hence, $g(x_i,x_j,\alpha)=g(x_j,r,\alpha)=(x_j+r)^{\alpha}+x_{j}^{\alpha}-(x_j+r-1)^{\alpha}-(x_j+1)^{\alpha}$. Suppose $r$ is a continuous variable, then we get $\frac{\partial g(x_j,r,\alpha)}{\partial r}=\alpha(x_j+r)^{\alpha-1}-\alpha(x_j+r-1)^{\alpha-1}$ is positive if $\alpha\in(-\infty,0)\cup(1,\infty)$, and negative if $\alpha\in(0,1)$. By the Lagrange mean-value theorem, we have
\begin{itemize}
	\item[\textnormal{(\romannumeral1)}] 
	\begin{align*}
		&~~~~x_{i}^{\alpha}+x_{j}^{\alpha}-(x_i-1)^{\alpha}-(x_j+1)^{\alpha}\\
		&\leq\left( x_{j}+2 \right)^{\alpha}+x_{j}^{\alpha}-2(x_j+1)^{\alpha}\\
		&=\alpha\xi_1^{\alpha-1}-\alpha\xi_2^{\alpha-1}\\
		&=\alpha(\alpha-1)(\xi_1-\xi_2)\eta^{\alpha-2}<0,
	\end{align*}  
	for $\alpha\in(0,1)$, where $x_j<\xi_2<x_j+1<\xi_1<x_i+2$, and $\xi_2<\eta<\xi_1$.
	\item[\textnormal{(\romannumeral2)}] \begin{align*}
		&~~~~x_{i}^{\alpha}+x_{j}^{\alpha}-(x_i-1)^{\alpha}-(x_j+1)^{\alpha}\\
		&\geq\left( x_{j}+2 \right)^{\alpha}+x_{j}^{\alpha}-2(x_j+1)^{\alpha}\\
		&=\alpha\xi_1^{\alpha-1}-\alpha\xi_2^{\alpha-1}\\
		&=\alpha(\alpha-1)(\xi_1-\xi_2)\eta^{\alpha-2}>0,
	\end{align*}  
	for $\alpha\in(-\infty,0)\cup(1,\infty)$, where $x_j<\xi_2<x_j+1<\xi_1<x_i+2$, and $\xi_2<\eta<\xi_1$.
\end{itemize}

This completes the proof.
\end{proof}

By repeating Lemma \ref{lem1}, finally, we can obtain the following corollary.
\begin{corollary}\label{coro1}
	Assume the function $f(x_1,x_2,\cdots,x_k)$ is defined as above. Then,
	\begin{itemize}
		\item[\textnormal{(\romannumeral1)}] for $\alpha\in(0,1)$, $f(x_1,x_2,\cdots,x_k)$ attains its maximum value if the difference between any two integer in $\left\lbrace x_1,x_2,\cdots,x_k\right\rbrace $ at most one, and
		\item[\textnormal{(\romannumeral2)}] for $\alpha\in(-\infty,0)\cup(1,\infty)$, $f(x_1,x_2,\cdots,x_k)$ attains its minimum value if the difference between any two integer in $\left\lbrace x_1,x_2,\cdots,x_k\right\rbrace $ at most one.
	\end{itemize}
\end{corollary}
Note that $D$ is a minimum dominating set in a tree $T$ with order $n$ and domination number $\gamma$, and $\overline{D}=V(T)\backslash D$. Let $E_1=\left\lbrace uv\in E_T|~u\in D,~v\in \overline{D} \right\rbrace $, $E_{2}=\left\lbrace uv\in E_T|~u\in D,~v\in D\right\rbrace $, $E_{3}=\left\lbrace uv\in E_T|~u\in \overline{D},~v\in \overline{D}\right\rbrace $ be three subsets of $E_T$, and $l_1=|E_1|$, $l_2=|E_2|$, $l_3=|E_3|$. It's obvious that
\begin{equation}\label{sys1}
	\begin{cases}
		l_1+l_2+l_3=|E_T|=n-1,\\
		\sum\limits_{v\in D}d_{v}=l_1+2l_2,\\
		\sum\limits_{v'\in\overline{D}}d_{v'}=l_1+2l_3.
	\end{cases}
\end{equation}

 Now the zeroth-order general Randi\'c index $^0R_{\alpha}(T)$ can be given by
\begin{equation}\label{sum1}
	^0R_{\alpha}(T)=\sum\limits_{v\in D}(d_{v})^{\alpha}+\sum\limits_{v'\in\overline{D}}(d_{v'})^{\alpha}.
\end{equation}

Since each $v\in\overline{D}$ is adjacent to at least one vertex of $D$, one can see that $l_1\geq n-\gamma$. Then by calculation, we obtain $l_2+l_3\leq\gamma-1$, implying
\begin{equation}\label{eqbasis}
	 |l_2-l_3|\leq\gamma-1.
\end{equation}

If $\alpha\in(0,1)$, then by Corollary \ref{coro1}, we can see the sum \eqref{sum1} necessarily attains maximum when degrees $d_{v}\in\left\lbrace\lfloor\frac{l_1+2l_2}{\gamma}\rfloor, \lceil\frac{l_1+2l_3}{\gamma}\rceil \right\rbrace $ for any vertex $v\in D$ and degrees $d_{v'}\in\left\lbrace\lfloor\frac{l_1+2l_2}{n-\gamma}\rfloor, \lceil\frac{l_1+2l_3}{n-\gamma}\rceil \right\rbrace$ for any vertex $v'\in\overline{D}$. Therefore, we let $l_1+2l_2=q\gamma+t$ ($0\leq t\leq \gamma-1$) and $l_1+2l_3=q'(n-\gamma)+t'$ $(0\leq t'\leq n-\gamma-1)$, where $q=\lfloor\frac{l_1+2l_2}{\gamma}\rfloor$, $t=l_1+2l_2-\gamma\lfloor\frac{l_1+2l_2}{\gamma}\rfloor$, $q'=\lfloor\frac{l_1+2l_3}{n-\gamma}\rfloor$ and $t'=l_1+2l_3-\gamma\lfloor\frac{l_1+2l_3}{n-\gamma}\rfloor$. Bearing in mind previous discussion, one can check that the formula shown in \eqref{sum1} will attain its maximum if $D$ contains $t$ vertices with degree $q+1$ and $\gamma-t$ vertices with degree $q$, and $\overline{D}$ contains $t'$ vertices with degree $q'+1$ and $n-\gamma-t'$ vertices with degree $q'$. Thus, we have
\begin{align*}
	\sum\limits_{v\in D}(d_{v})^{\alpha}&\leq t(q+1)^{\alpha}+(\gamma-t)q^{\alpha}\\
	&=\left(n-1+l_2-l_3-\gamma\lfloor\frac{n-1+l_2-l_3}{\gamma}\rfloor \right)\left[\left( \lfloor\frac{n-1+l_2-l_3}{\gamma}\rfloor+1\right)^{\alpha} -\left( \lfloor\frac{n-1+l_2-l_3}{\gamma}\rfloor\right)^{\alpha} \right]\\
	&~~~~+\gamma\left( \lfloor\frac{n-1+l_2-l_3}{\gamma}\rfloor\right)^{\alpha},
\end{align*}
and
\begin{align*}
	\sum\limits_{v'\in \overline{D}}(d_{v'})^{\alpha}&\leq t'(q'+1)^{\alpha}+(\gamma-t')(q')^{\alpha}\\
	&=\left[n-1+l_2-l_3-(n-\gamma)\lfloor\frac{n-1+l_3-l_2}{n-\gamma}\rfloor \right]\left[\left( \lfloor\frac{n-1+l_3-l_2}{n-\gamma}\rfloor+1\right)^{\alpha} -\left( \lfloor\frac{n-1+l_3-l_2}{n-\gamma}\rfloor\right)^{\alpha} \right]\\
	&~~~~+(n-\gamma)\left( \lfloor\frac{n-1+l_3-l_2}{n-\gamma}\rfloor\right)^{\alpha},
\end{align*}
which implies that
\begin{equation}\label{ineq1}
\begin{split}
	^0R_{\alpha}&\leq\left(n-1+l_2-l_3-\gamma\lfloor\frac{n-1+l_2-l_3}{\gamma}\rfloor \right)\left[\left( \lfloor\frac{n-1+l_2-l_3}{\gamma}\rfloor+1\right)^{\alpha} -\left( \lfloor\frac{n-1+l_2-l_3}{\gamma}\rfloor\right)^{\alpha} \right]\\
	&~~~+\left[n-1+l_2-l_3-(n-\gamma)\lfloor\frac{n-1+l_3-l_2}{n-\gamma}\rfloor \right]\left[\left( \lfloor\frac{n-1+l_3-l_2}{n-\gamma}\rfloor+1\right)^{\alpha} -\left( \lfloor\frac{n-1+l_3-l_2}{n-\gamma}\rfloor\right)^{\alpha} \right]\\
	&~~~+\gamma\left( \lfloor\frac{n-1+l_2-l_3}{\gamma}\rfloor\right)^{\alpha}+(n-\gamma)\left( \lfloor\frac{n-1+l_3-l_2}{n-\gamma}\rfloor\right)^{\alpha}. 
\end{split}	
\end{equation}

For fixed $n$ and $\gamma$, the right-hand side of the inequality \eqref{ineq1} can be viewed as the function $h(l_2-l_3)$, i.e. $^0R_{\alpha}\leq h(l_2-l_3)$ for $\alpha\in(0,1)$. 

Analogously, if $\alpha\in(-\infty,0)\cup(1,\infty)$, we can derive the inequality $^0R_{\alpha}\geq h(l_2-l_3)$. So far, we have obtained a simplified formula of bounds on $^0R_{\alpha}$.

\section{Bounds for the $^0R_{\alpha\in(0,1)}$ of trees in terms of domination number}
In this section, several upper and lower bounds for the zeroth-order general Randi\'c index $^0R_{\alpha\in(0,1)}$ of trees are determined. To characterize extremal $n$-vertex trees of $^0R_{2}$ with a given domination number $\gamma$, Borovi\'canin and Furtula \cite{Boro2} defined three trees family, denoted by $\mathcal{F}_{1}(n,\gamma)$, $\mathcal{F}_{2}(n,\gamma)$ and $\mathcal{F}_{3}(n,\gamma)$ in here, which will be shown below.

\begin{definition}
	\begin{itemize}
		\item[\textnormal{(\romannumeral1)}] $\mathcal{F}_{1}(n,\gamma)$ is a graph family contains some $n$-vertex trees with domination number $\gamma$ which consists of the stars of orders $\lfloor\frac{n-\gamma}{\gamma}\rfloor$ and $\lceil\frac{n-\gamma}{\gamma}\rceil$ with exactly $\gamma-1$ pairs of adjacent pendent vertices in neighbouring stars.
		\item[\textnormal{(\romannumeral2)}] $\mathcal{F}_{2}(n,\gamma)$ is a graph family contains some $n$-vertex trees $T$ with domination number $\gamma$ such that each vertex in $V_T$ has at most one pendent neighbour and $T$ satisfies: (1) there exists a minimum dominating set $D$ of $T$ has $3\gamma-n-2$ vertices with degree $3$ and $2(n-2\gamma)$ vertices with degree $2$, while $\overline{D}$ has $n-2\gamma+2$ vertices with degree $2$ and $3\gamma-n$ pendent vertices, or (2) there exists a minimum dominating set $D$ of $T$ has $n-2\gamma$ vertices with degree $2$ and $3\gamma-n$ pendent vertices, while $\overline{D}$ has $2(n-2\gamma+1)$ vertices with degree $2$, $3\gamma-n-2$ with degree $3$, and each vertex in $\overline{D}$ has only one neighbour in domination set $D$.
		\item[\textnormal{(\romannumeral3)}] $\mathcal{F}_{3}(n,\gamma)$ is a set of trees with order $n$ and domination number $\gamma$, which are obtained from the star $S_{n-\gamma+1}$ by attaching a pendant edge to its $\gamma-1$ pendent vertices.
	\end{itemize}
\end{definition}

\begin{figure}[htbp]
	\centering 
	\includegraphics[height=6cm, width=12cm]{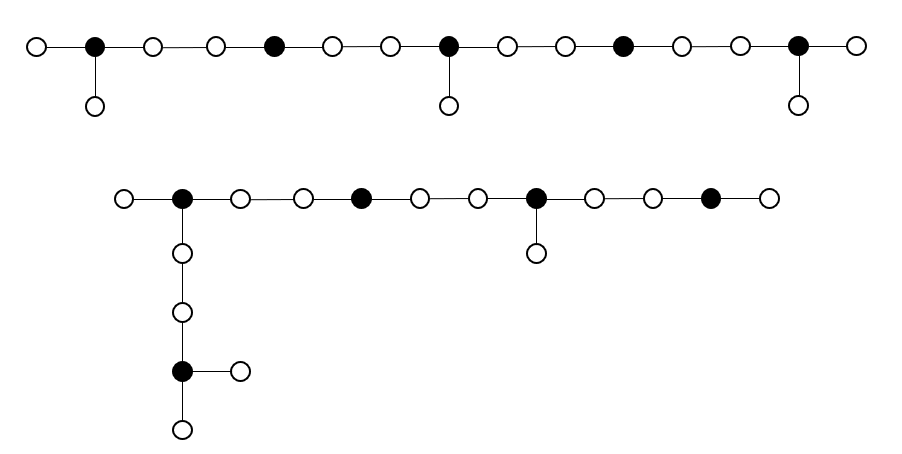}
	\caption{Tow non-isomorphic trees in the graph family $\mathcal{F}_1(n,\gamma)$.}
	\label{figure1}
\end{figure}

Next, we will give two theorems to prove that graph family $\mathcal{F}_{i}(n,\gamma)$ ($i=1,2,3$) are also extremal trees of $^0R_{\alpha}$ with a given domination number $\gamma$, for $\alpha\in(0,1)$.

\begin{theorem}
	Let $T$ be an $n$-vertex tree with domination number $\gamma$, $\alpha\in(0,1)$, then
	\begin{itemize}
		\item[\textnormal{(\romannumeral1)}]
		\begin{equation*}
			^0R_{\alpha}\leq\left[\left(\lfloor\frac{n-1}{\gamma}\rfloor\right)^{\alpha}-\left(\lfloor\frac{n-1}{\gamma}\rfloor-1\right)^{\alpha}  \right]\left(n-\gamma\lfloor\frac{n-1}{\gamma}\rfloor \right)+\gamma\left(\lfloor\frac{n-1}{\gamma}\rfloor-1\right)^{\alpha}+2(2^{\alpha}-1)(\gamma-1)+(n-\gamma),
		\end{equation*}  
		for $1\leq\gamma\leq\frac{n}{3}$, with equality holding if and only if $T\in\mathcal{F}_{1}(n,\gamma)$.
		\item[\textnormal{(\romannumeral2)}]
		\begin{equation*}
			^0R_{\alpha}\leq\begin{cases}
				(n-2)\cdot2^{\alpha}+2, & \mbox{for }\gamma = \lceil\frac{n}{3}\rceil, \\
				(-3^{\alpha}+3\cdot 2^{\alpha}-1) n+3(3^{\alpha}-2\cdot2^{\alpha}+1)\gamma+2(2^{\alpha}-3^{\alpha}), & \mbox{for } \frac{n+3}{3}\leq\gamma\leq\frac{n}{2},
			\end{cases}
		\end{equation*}
		with equality holding if and only if $T\in\mathcal{F}_{2}(n,\gamma)$.
	\end{itemize}
\end{theorem}
\begin{proof}
	(i). For path $P_3$, the theorem holds. Suppose $n\geq3$, then by $1\leq\gamma\leq\frac{n}{3}$, we get $n-\gamma\geq\frac{2n}{3}$, i.e. $\frac{\gamma-1}{n-\gamma}\leq\frac{n-3}{2n}<\frac{1}{2}$. Combining \eqref{eqbasis}, yields
	\begin{equation*}
		1=\frac{n-1-\gamma+1}{n-\gamma}\leq\frac{n-1+l_3-l_2}{n-\gamma}\leq\frac{n-1+\gamma-1}{n-\gamma}=1+2\frac{\gamma-1}{n-\gamma}<2,
	\end{equation*}
	implying
	\begin{equation*}
		q'=\lfloor \frac{n-1+l_3-l_2}{n-\gamma}\rfloor=1.
	\end{equation*}

	Then by $\frac{n-1+l_2-l_3}{\gamma}\geq\frac{n-1-\gamma+1}{\gamma}=\frac{n-\gamma}{\gamma}\geq\frac{2n}{n}=2$, we can see $q=\lfloor\frac{n-1+l_2-l_3}{\gamma}\rfloor\geq2$. Now, the function $h(l_2-l_3)$ can be formula that
	\begin{equation}
		\begin{split}
			h(l_2-l_3)&=\left[(q+1)^{\alpha}-q^{\alpha}+1-2^{\alpha} \right](l_2-l_3)+(n-\gamma q -1 )\left[\left(q+1\right)^{\alpha}-q^{\alpha} \right]+\gamma q^{\alpha}\\
			&~~~~+(\gamma-1)(2^{\alpha}-1)+(n-\gamma).
		\end{split}
	\end{equation}

	There are two possible cases to be consider.
	
	\textbf{Case 1.} $0\leq l_2-l_3\leq\gamma-1$.
	In such a case, $\frac{n-1}{\gamma}\leq\frac{n-1+l_2-l_3}{\gamma}\leq\frac{n-1+\gamma-1}{\gamma}<\frac{n-1}{\gamma}+1$, implying
	\begin{align*}
		q=\lfloor \frac{n-1+l_2-l_3}{\gamma} \rfloor &= \lfloor \frac{n-1}{\gamma} \rfloor,~~~~~~~\mbox{for  }0\leq l_2-l_3\leq\gamma\lfloor \frac{n-1}{\gamma} \rfloor+\gamma-n,\\
		q=\lfloor \frac{n-1+l_2-l_3}{\gamma} \rfloor &= \lfloor \frac{n-1}{\gamma} \rfloor+1,~~\mbox{for  }\gamma\lfloor \frac{n-1}{\gamma} \rfloor+\gamma-n+1\leq l_2-l_3\leq\gamma-1.
	\end{align*}

		Due to $\alpha\in(0,1)$, easily, one can check that $h(l_2-l_3)$ always decrease in interval $\left[ 0 ,\gamma\lfloor \frac{n-1}{\gamma} \rfloor+\gamma-n\right]$ and $\left[ \gamma\lfloor \frac{n-1}{\gamma} \rfloor+\gamma-n+1,\gamma-1\right]$. Thus, $h(l_2-l_3)$ will attain its maximum if $l_2-l_3=0$ or $l_2-l_3=\gamma\lfloor \frac{n-1}{\gamma} \rfloor+\gamma-n+1$. We consider the following difference
		\begin{equation}\label{dif1}
			h\left(\gamma\lfloor \frac{n-1}{\gamma} \rfloor+\gamma-n+1\right)-h(0)=\left[ \gamma\lfloor\frac{n-1}{\gamma}\rfloor+\gamma-(n-1) \right]\left[\left(\lfloor\frac{n-1}{\gamma}\rfloor+1 \right)^{\alpha}- \left(\lfloor\frac{n-1}{\gamma}\rfloor \right)^{\alpha}+1-2^{\alpha} \right].
		\end{equation}
	
	See $\frac{n-1}{\gamma}\geq\frac{n-1}{n/3}\geq2+\frac{n-3}{n}$, we have $\lfloor \frac{n-1}{\gamma} \rfloor\geq2$. Hence, 
	\begin{equation}\label{dif2}
		\left[\left(\lfloor\frac{n-1}{\gamma}\rfloor+1 \right)^{\alpha}- \left(\lfloor\frac{n-1}{\gamma}\rfloor \right)^{\alpha}+1-2^{\alpha} \right]\leq 3^{\alpha}+1-2^{\alpha}<0, 
	\end{equation}
	for any $\alpha\in(0,1)$. Combining \eqref{dif1} and \eqref{dif2}, yields $h\left( \gamma\lfloor \frac{n-1}{\gamma} \rfloor+\gamma-n+1\right)<h(0) $. Then the function $h(0)$ becomes
	\begin{equation}
		h(0)=\left(n-\gamma\lfloor \frac{n-1}{\gamma} \rfloor -1\right)\left[\left( \lfloor \frac{n-1}{\gamma} \rfloor+1\right)^{\alpha}-\left( \lfloor \frac{n-1}{\gamma} \rfloor\right)^{\alpha}   \right]+\gamma\left( \lfloor \frac{n-1}{\gamma} \rfloor\right)^{\alpha}+(n-\gamma)+(\gamma-1)(2^{\alpha}-1). 
	\end{equation}
	
	\textbf{Case 2.} $1-\gamma\leq l_2-l_3\leq 0$. Note that $\frac{n-\gamma-1}{\gamma}\leq\frac{n-\gamma}{\gamma}\leq\frac{n-1+l_2-l_3}{\gamma}\leq\frac{n-1}{\gamma}$. From \eqref{eqbasis}, we get
	\begin{align}
		q=\lfloor \frac{n-1+l_2-l_3}{\gamma} \rfloor &= \lfloor \frac{n-1}{\gamma} \rfloor,~~~~~~~\mbox{for  }\gamma\lfloor \frac{n-1}{\gamma} \rfloor-n+1\leq l_2-l_3\leq0,\label{theinequ3}\\
		q=\lfloor \frac{n-1+l_2-l_3}{\gamma} \rfloor &= \lfloor \frac{n-1}{\gamma} \rfloor-1,~~\mbox{for  }1-\gamma\leq l_2-l_3\leq\gamma\lfloor \frac{n-1}{\gamma} \rfloor-n.\label{theinequ4}
	\end{align}
	
	Analogously, we conclude that $h(l_2-l_3)$ will attain its maximum if $l_2-l_3=\gamma\lfloor \frac{n-1}{\gamma} \rfloor-n+1$ or $l_2-l_3=1-\gamma$. Consider the difference $h\left(\gamma\lfloor \frac{n-1}{\gamma} \rfloor-n+1 \right)-h(1-\gamma) $ which is given by
	\begin{equation}
		h\left(\gamma\lfloor \frac{n-1}{\gamma} \rfloor-n+1 \right)-h(1-\gamma)=-\left(n-\gamma\lfloor \frac{n-1}{\gamma} \rfloor -\gamma\right)\left[\left(\lfloor\frac{n-1}{\gamma}\rfloor \right)^{\alpha}- \left(\lfloor\frac{n-1}{\gamma}\rfloor-1 \right)^{\alpha}+1-2^{\alpha} \right].
	\end{equation}
	
	Since $ n-\gamma-\gamma\lfloor\frac{n-1}{\gamma}\rfloor\leq0 $, for $\gamma\geq2$ and $\left(\lfloor\frac{n-1}{\gamma}\rfloor \right)^{\alpha}- \left(\lfloor\frac{n-1}{\gamma}\rfloor-1 \right)^{\alpha}+1-2^{\alpha} \leq0$, for $\alpha\in(0,1)$ and $\lfloor\frac{n-1}{\gamma}\rfloor\geq2$, we have $h\left(\gamma\lfloor \frac{n-1}{\gamma} \rfloor-n+1 \right)\leq h(1-\gamma)$. In addition, if $n=\gamma\lfloor\frac{n-1}{\gamma}\rfloor+\gamma$, then only the inequality \eqref{theinequ4} holds. Note that $\left(\lfloor\frac{n-1}{\gamma}\rfloor \right)^{\alpha}- \left(\lfloor\frac{n-1}{\gamma}\rfloor-1 \right)^{\alpha}=2^{\alpha}-1$ if and only if $\lfloor\frac{n-1}{\gamma}\rfloor=2$, implying, $2\gamma+1\leq n<3\gamma+1$. Hence, $n=3\gamma$. Then we can conclude that the extremal tree $T$ in such a case consist of $\gamma$ $3$-vertex paths $P_3$ with exactly $\gamma-1$ pairs of adjacent pendent vertices in neighbouring paths. Obviously, $T\in\mathcal{F}_{1}(n,\gamma)$. 
	
	Consequently, in order to find the feasible maximum value of $h(l_2-l_3)$, we just need to calculate the value of $h(0)-h(1-\gamma)$,
	\begin{equation*}
		\begin{split}
			h(0)-h(1-\gamma)&=\left[\left( \lfloor \frac{n-1}{\gamma} \rfloor+1\right)^{\alpha}-\left( \lfloor \frac{n-1}{\gamma} \rfloor\right)^{\alpha}   \right]\left(n-\gamma\lfloor \frac{n-1}{\gamma} \rfloor -1\right)+\gamma\left( \lfloor \frac{n-1}{\gamma} \rfloor\right)^{\alpha}\\
			&~~~~-\left[\left( \lfloor \frac{n-1}{\gamma} \rfloor\right)^{\alpha}-\left( \lfloor \frac{n-1}{\gamma} \rfloor-1\right)^{\alpha}  \right] \left(n-\gamma \lfloor \frac{n-1}{\gamma} \rfloor \right) -\gamma\left( \lfloor \frac{n-1}{\gamma} \rfloor-1\right)^{\alpha}\\
			&~~~~+(n-\gamma)+(\gamma-1)(2^{\alpha}-1)-2(2^{\alpha}-1)(\gamma-1)-(n-\gamma).
		\end{split}
	\end{equation*}

Note that $0<\left( \lfloor \frac{n-1}{\gamma} \rfloor+1\right)^{\alpha}-\left( \lfloor \frac{n-1}{\gamma} \rfloor\right)^{\alpha} <\left( \lfloor \frac{n-1}{\gamma} \rfloor\right)^{\alpha}-\left( \lfloor \frac{n-1}{\gamma} \rfloor-1\right)^{\alpha} $ for $\alpha\in(0,1)$. Then we have
\begin{align*}
		h(0)-h(1-\gamma)&=\left(n-\gamma \lfloor \frac{n-1}{\gamma} \rfloor-1 \right)\left[\left( \lfloor \frac{n-1}{\gamma} \rfloor+1\right)^{\alpha}+\left( \lfloor \frac{n-1}{\gamma} \rfloor-1\right)^{\alpha}-2\left( \lfloor \frac{n-1}{\gamma} \rfloor\right)^{\alpha} \right]\\
		&~~~~+(\gamma-1)\left[\left( \lfloor \frac{n-1}{\gamma} \rfloor\right)^{\alpha}-\left( \lfloor \frac{n-1}{\gamma} \rfloor-1\right)^{\alpha}+2^{\alpha}-1   \right]<0.
\end{align*}

The inequality is strict. Thus,
\begin{equation*}
	^0R_{\alpha}\leq\left[\left(\lfloor\frac{n-1}{\gamma}\rfloor\right)^{\alpha}-\left(\lfloor\frac{n-1}{\gamma}\rfloor-1\right)^{\alpha}  \right]\left(n-\gamma\lfloor\frac{n-1}{\gamma}\rfloor \right)+\gamma\left(\lfloor\frac{n-1}{\gamma}\rfloor-1\right)^{\alpha}+2(2^{\alpha}-1)(\gamma-1)+(n-\gamma),
\end{equation*}
for $\alpha\in(0,1)$ and $1\leq\gamma\leq\frac{n}{3}$. Equality holds if and only if $l_2-l_3=1-\gamma$. Combining \eqref{sys1} and \eqref{sum1}, yields $l_1=n-\gamma$, $l_2=0$ and $l_3=\gamma-1$. One can easily check that the corresponding extremal trees in such a case all belong to $\mathcal{F}_{1}(n,\gamma)$.

\begin{figure}[htbp]
	\centering 
	\includegraphics[height=4cm, width=12cm]{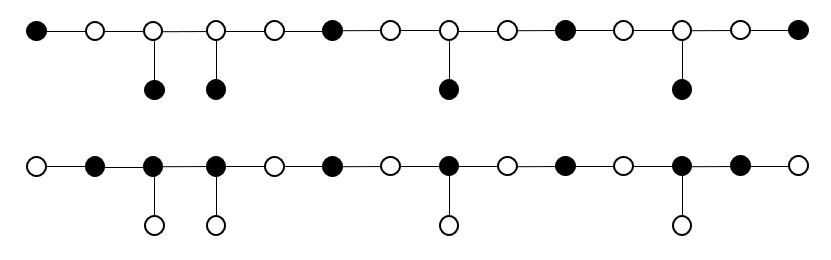}
	\caption{Tow non-isomorphic trees in the graph family $\mathcal{F}_2(n,\gamma)$.}
	\label{figure2}
\end{figure}

(ii). For $\gamma=\lceil\frac{n}{3}\rceil$, the path $P_n$ is the unique tree such that $^0R_{\alpha}$ with $\alpha\in(0,1)$ attains maximum. Then we suppose $\gamma\geq\frac{n+3}{3}$. Now, it holds $2\gamma\leq n\leq 3\gamma-3$, implying $n\geq6$ and $\gamma\geq3$. Due to
\begin{equation*}
	1=\frac{n-\gamma}{n-\gamma}\leq\frac{n-1+l_3-l_2}{n-\gamma}\leq\frac{n-1+\gamma-1}{n-\gamma}=1+2\frac{\gamma-1}{n-\gamma}<3,
\end{equation*}
which implies that $q'=\lfloor\frac{n-1+l_3-l_2}{n-\gamma}\rfloor=2$ or $q'=\lfloor\frac{n-1+l_3-l_2}{n-\gamma}\rfloor=1$. Therefore, we consider the following two cases.

\textbf{Case 1.} $q'=\lfloor\frac{n-1+l_3-l_2}{n-\gamma}\rfloor=1$.
In this case, one can see $l_3-l_2<n-2\gamma+1$, implying $l_2-l_3\geq 2\gamma-n$. Note that $\gamma\leq\frac{n}{2}$, thus, $2\gamma-n\leq0$. For convince, we will divide the discussion into $2\gamma-n\leq-1$ and $2\gamma-n=0$.

\textbf{Case 1.1.} $2\gamma-n\leq-1$.
Obviously, $2\leq\frac{n-1}{\gamma}\leq\frac{n-1}{n/3+1}<3$, then we have $\lfloor\frac{n-1}{\gamma}\rfloor=2$. Assume that $2\gamma-n\leq l_2-l_3\leq0$. Similarly, one can see
\begin{align}
	q=\lfloor \frac{n-1+l_2-l_3}{\gamma} \rfloor &= \lfloor \frac{n-1}{\gamma} \rfloor=2,~~~~~~~\mbox{for }2\gamma-n+1\leq l_2-l_3\leq0,\label{theinequ5}\\
	q=\lfloor \frac{n-1+l_2-l_3}{\gamma} \rfloor &= \lfloor \frac{n-1}{\gamma} \rfloor-1=1,~~\mbox{for }l_2-l_3=2\gamma-n.\label{theinequ6}
\end{align}

Due to $q'=\lfloor\frac{n-1+l_3-l_2}{n-\gamma}\rfloor=1$ and $\gamma\geq\frac{n+3}{3}$, then the only relation \eqref{theinequ5} holds. Hence,
\begin{equation}\label{hfuc1}
	\begin{split}
		h(l_2-l_3)&=(3^{\alpha}-2\cdot2^{\alpha}+1)(l_2-l_3)+(3^{\alpha}-2^{\alpha}+1)n\\
		&~~~~+(4\cdot2^{\alpha}-2\cdot3^{\alpha}-2)\gamma-(3^{\alpha}-1),  (2\gamma-n+1\leq l_2-l_3\leq 0).
	\end{split}
\end{equation}

Now, we suppose $0\leq l_2-l_3\leq\gamma-1$. Analogously, one can see
\begin{align*}
	q=\lfloor \frac{n-1+l_2-l_3}{\gamma} \rfloor &= \lfloor \frac{n-1}{\gamma} \rfloor=2,~~~~~~~\mbox{for  }0\leq l_2-l_3\leq 3\gamma-n,\\
	q=\lfloor \frac{n-1+l_2-l_3}{\gamma} \rfloor &= \lfloor \frac{n-1}{\gamma} \rfloor+1=3,~~\mbox{for  }3\gamma-n+1\leq l_2-l_3\leq \gamma-1,
\end{align*}
implying
\begin{equation}\label{hfuc2}
	\begin{split}
		h(l_2-l_3)&=(3^{\alpha}-2\cdot2^{\alpha}+1)(l_2-l_3)+(3^{\alpha}-2^{\alpha}+1)n\\
		&~~~~+(4\cdot2^{\alpha}-2\cdot3^{\alpha}-2)\gamma-(3^{\alpha}-1),~~(0\leq l_2-l_3\leq 3\gamma-n)
	\end{split}
\end{equation}
and
\begin{equation*}
	\begin{split}	h(l_2-l_3)&=(4^{\alpha}-3^{\alpha}-2^{\alpha}+1)(l_2-l_3)+(4\cdot3^{\alpha}-3\cdot4^{\alpha}+2^{\alpha}-2)\gamma\\
	&~~~~+(4^{\alpha}-3^{\alpha}+1)n-(4^{\alpha}-3^{\alpha}+2^{\alpha}-1),~~(3\gamma-n+1\leq l_2-l_3\leq\gamma-1)
	\end{split}
\end{equation*}

Then by \eqref{hfuc1} and \eqref{hfuc2}, we obtain
\begin{equation}\label{hfuc3}
	\begin{split}
		h(l_2-l_3)&=(3^{\alpha}-2\cdot2^{\alpha}+1)(l_2-l_3)+(3^{\alpha}-2^{\alpha}+1)n\\
		&~~~~+(4\cdot2^{\alpha}-2\cdot3^{\alpha}-2)\gamma-(3^{\alpha}-1),  (2\gamma-n+1\leq l_2-l_3\leq 3\gamma-n).
	\end{split}
\end{equation}

See $h(3\gamma-n+1)-h(3\gamma-n)=(3^{\alpha}-2\cdot2^{\alpha}+1)<0$ for any $\alpha\in(0,1)$, we just need consider the relation \eqref{hfuc3}. 

Note that $^0R_{\alpha\in(0,1)}(T)$ attains its maximum if and only if $T$ is a path (See \cite{Li} Theorem 4.2), we conclude that an extremal $T$, whose $^0R_{\alpha\in(0,1)}$ is maximum, only consist of vertices with degree $1$, $2$ and $3$. To determine a sharp upper bound on $^0R_{\alpha}$, we must investigate further to find a feasible value of $l_2-l_3$. For a minimum dominating set $D$ of tree $T$,  the number of vertices with degree $2$ and $3$ are denoted by $s_2$ and $s_3$, respectively. Also, for the set $\overline{D}$, the number of vertices with degree $1$ and $2$are denoted by $s'_{1}$ and $s'_{2}$, respectively. It's holds
\begin{equation}\label{system1}
	\begin{cases}
		|V_T|=s_2+s_3+s'_1+s'_2,\\
		s_2+s_3=\gamma,\\
		s'_1+s'_2=n-\gamma.
	\end{cases}
\end{equation}

Combining $\sum_{v\in V_T}d_v=2(n-1)=2(s_2+s_3+s'_1+s'_2-1)=s'_1+2(s_2+s'_2)+3s_3$, yields $s_3=s'_1-2$ and $s_2-s'_2=2\gamma-n+2$. From \eqref{system1}, we get
\begin{equation}\label{system2}
	\begin{cases}
		n-1+l_2-l_3=2s_2+3s'_1-6,\\
		n-1+l_3-l_2=s'_1+2s'_2.
	\end{cases}
\end{equation}

Based on \eqref{ineq1} and system \eqref{system2}, the function $h(l_2-l_3)$ becomes
\begin{equation*}
	h(s'_1)=(3^{\alpha}-2\cdot2^{\alpha}+1)s'_1+2^{\alpha}\cdot n+2(2^{\alpha}-3^{\alpha}), ~~\mbox{for }2\leq s'_1\leq\gamma+1.
\end{equation*}

\textbf{Case 1.2.} $2\gamma-n=0$, i.e., $\gamma=\frac{n}{2}$ if $n$ is even.

Then, $1\leq\frac{n-1}{\gamma}=1+\frac{\gamma-1}{\gamma}<2$, which implies $\lfloor\frac{n-1}{\gamma}\rfloor=1$. One can easily check that the following relations hold.
\begin{align*}
	q=\lfloor \frac{n-1+l_2-l_3}{\gamma} \rfloor &= \lfloor \frac{n-1}{\gamma} \rfloor=1,~~~~~~~\mbox{for  } l_2-l_3=0,\\
	q=\lfloor \frac{n-1+l_2-l_3}{\gamma} \rfloor &= \lfloor \frac{n-1}{\gamma} \rfloor+1=2,~~\mbox{for  }1\leq l_2-l_3\leq\frac{n-2}{2}.
\end{align*}

By the analogous derivation, the function $h(l_2-l_3)$ can be given by
\begin{equation*}
	h(s'_1)=(3^{\alpha}-2\cdot2^{\alpha}+1)s'_1+2^{\alpha}\cdot n+2(2^{\alpha}-3^{\alpha}), ~~\mbox{for }2\leq s'_1\leq\frac{n}{2}.
\end{equation*}

In \cite{Boro2}, Borovi\'canin and Furtula have proved that $s'_1\geq 3\gamma-n$ for trees, and $s'_1>3\gamma-n$ always holds if there exists a vertex in $V_T$ has two pendent neighbours. It's obvious that $^0R_{\alpha}$ of trees attains its maximum value if $s'_1=3\gamma-n$, i.e., $l_2-l_3=5\gamma-2n+1$. In such a case, one can check that corresponding extremal trees all belong to $\mathcal{F}_2(n,\gamma)$. Now, the function $h(l_2-l_3)$ becomes
\begin{align*}
	h(3\gamma-n)&=(3^{\alpha}-2\cdot2^{\alpha}+1)(3\gamma-n)+2^{\alpha}\cdot n+2(2^{\alpha}-3^{\alpha})\\
	&=(-3^{\alpha}+3\cdot 2^{\alpha}-1) n+3(3^{\alpha}-2\cdot2^{\alpha}+1)\gamma+2(2^{\alpha}-3^{\alpha}).
\end{align*}

\textbf{Case 2.} $q'=\lfloor\frac{n-1+l_3-l_2}{n-\gamma}\rfloor=2$.
Due to $2\leq\frac{n-1+l_3-l_2}{n-\gamma}<3$, we have $l_2-l_3\leq2\gamma-n+1<0$. It holds
\begin{equation*}
	1\leq\frac{n-\gamma}{\gamma}\leq\frac{n-1+l_2-l_3}{\gamma}\leq\frac{2(\gamma-1)}{\gamma}<2,
\end{equation*}
implying $q=\lfloor \frac{n-1+l_2-l_3}{\gamma} \rfloor=1$. If $l_3-l_2=n-2\gamma+1$, we have $\frac{n-1+l_3-l_2}{n-\gamma}=2$, implying all vertices in $\overline{D}$ has degrees 2, where $D$ is an arbitrary minimum dominating set. Consequently, all vertices in $D$ have degree 1 and 2, i.e., $T\cong P_n$, a contradiction, since $\gamma\geq\frac{n+3}{3}$. 

Next, we assume $l_3-l_2\geq n-2\gamma+2$, then by \eqref{ineq1}, we get
\begin{equation*}
	\begin{split}
		h(l_2-l_3)&=(-3^{\alpha}+2\cdot2^{\alpha}-1)(l_2-l_3)+(3\cdot2^{\alpha}-3^{\alpha}-1)n\\
		&~~~~+(2\cdot3^{\alpha}-4\cdot2^{\alpha}+2)\gamma+(1-3^{\alpha}),~~(1-\gamma\leq l_2-l_3\leq 2\gamma-n-2).
	\end{split}
\end{equation*}

Analogously, we have to find the minimum realizable value of $l_2-l_3$. For an arbitrary minimum dominating set $D$ of $T$, the number of vertices with degree 1, and 2 are denoted by $s_1$ and $s_2$, respectively, and for the set $\overline{D}$, the number of vertices with degree 2, and 3 are denoted by $s'_2$ and $s'_3$, respectively. Apparently, it holds $s_2-s'_2=2\gamma-n-2$ and $l_2-l_3=2\gamma-n-s_1+1$. Hence, the function $h(l_2-l_3)$ can be given by
\begin{equation*}
	h(s_1)=(3^{\alpha}-2\cdot2^{\alpha}+1)s_1+2^{\alpha}n+2(2^{\alpha}-3^{\alpha}), \mbox{ for }3\leq s_1\leq 3\gamma-n.
\end{equation*}

Based on previous discussions, we can determined the only possible value of $s_1$, that is $3\gamma-n$, implying $l_2-l_3=1-\gamma$, such that there exists a corresponding extremal trees with order $n$ and domination number $\gamma$, where $\frac{n+3}{3}\leq\gamma\leq\frac{n}{2}$, satisfying its all vertices in an arbitrary minimum dominating set $D$ have degrees 1 and 2, while all vertices in $\overline{D}$ have degrees 2 and 3. Then, the function $h(l_2-l_3)$ can be written as
\begin{align*}
	h(3\gamma-n)&=(3^{\alpha}-2\cdot2^{\alpha}+1)(3\gamma-n)+2^{\alpha}\cdot n+2(2^{\alpha}-3^{\alpha})\\
	&=(-3^{\alpha}+3\cdot 2^{\alpha}-1) n+3(3^{\alpha}-2\cdot2^{\alpha}+1)\gamma+2(2^{\alpha}-3^{\alpha}).
\end{align*}

At this time, we have $l_1=n-\gamma$, $l_2=0$, and $l_3=\gamma-1$. Based on previous considerations, one can check that the corresponding extremal trees all belong to $\mathcal{F}_2(n,\gamma)$.

This complete the proof.
\end{proof}

\begin{figure}[htbp]
	\centering 
	\includegraphics[height=3.2cm, width=8.4cm]{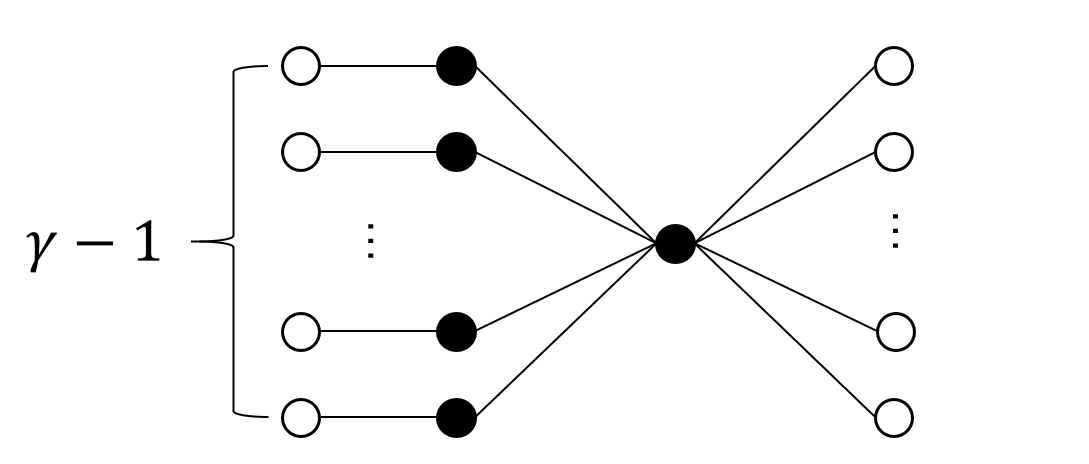}
	\caption{The tree family $\mathcal{F}_3(n,\gamma)$.}
	\label{figure3}
\end{figure}
\begin{theorem}
	Let $T$ be an $n$-vertex tree with domination number $\gamma$, $\alpha\in(0,1)$, then
	\begin{equation}\label{inequthe}
		^0R_{\alpha}\geq(n-\gamma)^{\alpha}+(n-\gamma)+(\gamma-1)\cdot2^{\alpha},
	\end{equation}
	with equality holding if and only if $T\in\mathcal{F}_{3}(n,\gamma)$.
\end{theorem}
\begin{proof}
	Assume first that $\Delta(T)=2$, implying $T$ is a path and $\gamma(T)=\lceil\frac{n}{3}\rceil$. It holds $T_2\in\mathcal{F}_3(2,1)$ ($\cong P_2$), $T_3\in\mathcal{F}_3(3,1)$ ($\cong P_3$) and $T_4\in\mathcal{F}_3(4,2)$ ($\cong P_4$). If $n\geq5$, the inequality in \eqref{inequthe} is strict.
	
	For $\Delta(T)\geq3$, we take an arbitrary diameter in $T$, denoted by $v_1v_2\dots v_d$ ($d\geq4$). From the definition of domination number, it holds $\Delta(T)\leq n-\gamma$. Then, we prove this theorem by induction on $n$. Suppose the inequality in \eqref{inequthe} holds for $|V_T|=n-1$.  If $|V_T|=n$, we let $T_{-1}=T-\left\lbrace v_1 \right\rbrace $ and consider the following two cases.
	
	\textbf{Case 1.} $\gamma(T_{-1})=\gamma(T)$.
	By calculation, we get
	\begin{align}
		^0R_{\alpha}(T)&={^0}R_{\alpha}(T_{-1})-\left( d_{v_2}-1\right) ^{\alpha}+\left( d_{v_2} \right)^{\alpha}+1\nonumber\\
		&\geq (n-\gamma-1)+(n-\gamma-1)^{\alpha}+(\gamma-1)\cdot2^{\alpha}-\left( d_{v_2}-1\right) ^{\alpha}+\left( d_{v_2} \right)^{\alpha}+1\label{inequ316}\\
		&=(n-\gamma)+(n-\gamma)^{\gamma}+(\gamma-1)\cdot2^{\alpha}+\left[ (n-\gamma-1)^{\alpha}-(n-\gamma)^{\alpha}\right]-\left[(d_{v_2}-1)^{\alpha}-(d_{v_2})^{\alpha}\right]\nonumber\\
		&\geq (n-\gamma)^{\alpha}+(n-\gamma)+(\gamma-1)\cdot2^{\alpha},~~\alpha\in(0,1).\label{inequ317}
	\end{align}

	Equalities in both \eqref{inequ316} and \eqref{inequ317} hold if and only if $d_{v_2}=\Delta(T)=n-\gamma$, implying $T\in\mathcal{F}_{3}(n,\gamma)$.
	
	\textbf{Case 2.} $\gamma(T_{-1})=\gamma(T)-1$.
	Combining the definition of domination set, yields $d_{v_2}=2$. Hence,
	\begin{align}
		^0R_{\alpha}(T)&={^0}R_{\alpha}(T_{-1})-\left( d_{v_2}-1\right) ^{\alpha}+\left( d_{v_2} \right)^{\alpha}+1\nonumber\\
		&\geq (n-\gamma)^{\alpha}+(n-\gamma)+(\gamma-1)\cdot2^{\alpha}+\left[(d_{v_2})^{\alpha}-(d_{v_2}-1)^{\alpha}+1-2^{\alpha} \right]\label{inequ318}\\
		&=(n-\gamma)^{\alpha}+(n-\gamma)+(\gamma-1)\cdot2^{\alpha},~~\alpha\in(0,1).\nonumber
	\end{align}

	Equality in \eqref{inequ318} holds if and only if $T_{-1}\in\mathcal{F}_{3}(n-1,\gamma-1)$, i.e., $T\in\mathcal{F}_{3}(n,\gamma)$.
	
	This completes the proof.
\end{proof}

\section{Bounds for the $^0R_{\alpha\in(-\infty,0)\cup(1,\infty)}$ of trees in terms of domination number}
Next, we will show two theorems to clarify the sharp upper and lower bounds on $^0R_{\alpha\in(-\infty,0)\cup(1,\infty)}$. Based on the Corollary \eqref{coro1} and proofs in Section 3, we can derive the following two theorems by a straightforward procedure.

\begin{theorem}\label{them41}
	Let $T$ be an $n$-vertex tree with domination number $\gamma$, $\alpha\in(-\infty,0)\cup(1,\infty)$, then
	\begin{itemize}
		\item[\textnormal{(\romannumeral1)}]
		\begin{equation*}
			^0R_{\alpha}\geq\left[\left(\lfloor\frac{n-1}{\gamma}\rfloor\right)^{\alpha}-\left(\lfloor\frac{n-1}{\gamma}\rfloor-1\right)^{\alpha}  \right]\left(n-\gamma\lfloor\frac{n-1}{\gamma}\rfloor \right)+\gamma\left(\lfloor\frac{n-1}{\gamma}\rfloor-1\right)^{\alpha}+2(2^{\alpha}-1)(\gamma-1)+(n-\gamma),	
		\end{equation*}  
		for $1\leq\gamma\leq\frac{n}{3}$, with equality holding if and only if $T\in\mathcal{F}_{1}(n,\gamma)$.
		\item[\textnormal{(\romannumeral2)}] \begin{equation*}
			^0R_{\alpha}\geq\begin{cases}
				(n-2)\cdot2^{\alpha}+2, & \mbox{for }\gamma = \lceil\frac{n}{3}\rceil, \\
				(-3^{\alpha}+3\cdot 2^{\alpha}-1) n+3(3^{\alpha}-2\cdot2^{\alpha}+1)\gamma+2(2^{\alpha}-3^{\alpha}), & \mbox{for } \frac{n+3}{3}\leq\gamma\leq\frac{n}{2},
			\end{cases}
		\end{equation*} 
		with equality holding if and only if $T\in\mathcal{F}_{2}(n,\gamma)$.
	\end{itemize}
\end{theorem}

\begin{theorem}\label{them42}
	Let $T$ be an $n$-vertex tree with domination number $\gamma$, $\alpha\in(-\infty,0)\cup(1,\infty)$, then
	\begin{equation*}
		^0R_{\alpha}\leq(n-\gamma)^{\alpha}+(n-\gamma)+(\gamma-1)\cdot2^{\alpha},
	\end{equation*}
	with equality holding if and only if $T\in\mathcal{F}_{3}(n,\gamma)$.
\end{theorem}

Combining the results in \cite{Boro2} and \cite{CLiu}, Theorem \ref{them41} and \ref{them42} can be directly verified.

\section*{Acknowledgements} The authors would like to express their sincere gratitude to all the referees for their careful reading and insightful suggestions.

\end{document}